\documentclass[a4paper]{amsart}
        \usepackage{latexsym}
        \usepackage{amssymb}
        \usepackage{amsmath}
        \usepackage{amsfonts}
        \usepackage{amsthm}
        \usepackage[hypertexnames=false]{hyperref}
        \usepackage[all,knot,poly,2cell]{xy}
             \UseAllTwocells
        \usepackage{xspace}

        \usepackage{eucal}

        \theoremstyle{plain}
        \newtheorem{theorem}{Theorem}[section]
        \newtheorem{corollary}[theorem]{Corollary}
        \newtheorem{lemma}[theorem]{Lemma}
        \newtheorem{proposition}[theorem]{Proposition}

        \theoremstyle{definition}

        \theoremstyle{remark}
        \newtheorem{remark}[theorem]{Remark}
        \newcommand{\suchthat}{\,:\,}
        \newcommand{\itemref}[1]{\eqref{#1}}

        \newcommand{\mathscript}{\mathcal}

        \newcommand{\Z}{\mathbb{Z}}

        \newcommand{\Q}{\mathbb{Q}}

        \newcommand{\Orb}{\mathcal{O}}   
       \DeclareMathOperator{\spec}{Spec} 
        \newcommand{\red}[1]{{#1}_{\mathrm{red}}} 

           \newcommand{\lisset}{\mathrm{lis\text{\nobreakdash-}\acute{e}t}}

        \newcommand{\MOD}{\mathsf{Mod}}    

         \DeclareMathOperator{\Hom}{Hom}
        \DeclareMathOperator{\Ext}{Ext}

        \newcommand{\trunc}[1]{\tau^{{#1}}}
        \newcommand{\RDERF}{\mathsf{R}}
        \newcommand{\LDERF}{\mathsf{L}}
        
        \newcommand{\DCAT}{\mathsf{D}}
        \newcommand{\RHom}{\RDERF\!\Hom}

        \newcommand{\QCOH}{\mathsf{QCoh}}
        \newcommand{\COH}{\mathsf{Coh}}

        \renewcommand{\bar}[1]{\overline{{#1}}}

        \newcommand{\tensor}{\otimes}
        
        \newcommand{\homotopic}{\simeq}

\newcommand{\fndefn}[1]{\emph{{#1}}}

\renewcommand{\subset}{\subseteq}
\numberwithin{equation}{section}

\newcommand{\qcsubscript}{\mathrm{qc}} 

\newcommand{\DQCOH}[1][]{\DCAT_{\qcsubscript{#1}}} 

\newcommand{\QCOHPSH}[1]{(#1_{\QCOH})_*} 
\newcommand{\MODPSH}[1]{(#1_{\lisset})_*} 

\newcommand{\Galpha}{\pmb{\alpha}} 
\newcommand{\GL}{\mathrm{GL}} 
\newcommand{\Ga}{\mathbb{G}_a} 

\newcommand{\KCAT}{\mathsf{K}}

\newcommand{\holim}[1]{\underset{#1}{\mathrm{holim}}\,}

\newcommand{\spref}[1]{\href{http://stacks.math.columbia.edu/tag/#1}{#1}}

\makeatletter
\newcommand{\labitem}[2]{%
\def\@itemlabel{(\textbf{#1})}
\item
\def\@currentlabel{\textbf{#1}}\label{#2}}
\makeatother
\newcommand\nc {\newcommand}

\theoremstyle{definition}

\newcommand{\zz}{{\mathbb Z}}

\nc\script{\mathscript}
\nc\z{\zeta}
\nc\bc{{\mathbb{BC}}}
\nc\ct{{\script T}}
\nc\cf{{\script F}}
\nc\cl{{\script L}}
\nc\cv{{\script V}}
\nc\ce{{\script E}}
\nc\ch{{\script H}}
\nc\cs{{\script S}}
\nc\car{{\script R}}
\nc\cd{{\script D}}
\nc\cc{{\script C}}
\nc\ck{{\script K}}
\nc\ca{{\script A}}
\nc\ci{{\script I}}
\nc\cj{{\script J}}
\nc\co{{\script O}}
\nc\cm{{\script M}}
\nc\colim{{\ds\mathop{\rm colim}_{\ds\la}}}
\nc\oa{\overline{\ca}}
\nc\s{\sigma}
\nc\ta{\tau}
\nc\os{\overline\sigma}
\nc\ot{\overline\tau}
\nc\T{\Sigma}
\nc\Tm{\Sigma^{-1}}

\nc\cb{{\script B}}
\nc\ab{{\script A}b}

\nc\id{\text{\rm id}}
\nc\Loc{\mathrm{Loc}}

\newcommand{\ls}{\mathrm{ls}}
\newcommand{\badgroup}{poor}

\newcommand{\badstack}{{\badgroup}ly stabilized}

\newcommand*\cocolon{%
        \nobreak
        \mskip6mu plus1mu
        \mathpunct{}%
        \nonscript
        \mkern-\thinmuskip
        {:}%
        \mskip2mu
        \relax
}

\newcommand{\REP}{\mathsf{Rep}}
\let\colim\relax
\DeclareMathOperator{\colim}{colim}
\begin{document}
\author[J. Hall]{Jack Hall}
\address{Mathematical Sciences Institute\\
        The Australian National University\\
        Acton, ACT, 2601\\
        Australia}
\email{jack.hall@anu.edu.au}
\author[A. Neeman]{Amnon Neeman}
\address{Mathematical Sciences Institute\\
        The Australian National University\\
        Acton, ACT, 2601\\
        Australia}
\email{Amnon.Neeman@anu.edu.au}
\author[D. Rydh]{David Rydh}
\address{KTH Royal Institute of Technology\\
         Department of Mathematics\\
         SE\nobreakdash-100\ 44\ Stockholm\\
         Sweden}
\email{dary@math.kth.se}
\thanks{The third author is supported by the Swedish Research Council (VR), grant number 2011-5599.}

\title[Derived categories of algebraic stacks]{One positive and two negative results for derived categories of algebraic stacks}
\date{December 3, 2015}
\begin{abstract}
  Let $X$ be a quasi-compact and quasi-separated scheme. There are two fundamental and pervasive facts about the unbounded derived category of $X$: (1) $\mathsf{D}_{\mathrm{qc}}(X)$ is compactly generated by perfect complexes and (2) if $X$ is noetherian or has affine diagonal, then the functor $\Psi_X \colon \mathsf{D}(\mathsf{QCoh}(X)) \to \mathsf{D}_{\mathrm{qc}}(X)$ is an equivalence. Our main results are that for algebraic stacks in positive characteristic, the assertions (1) and (2) are typically false.

\end{abstract}



\maketitle
\numberwithin{equation}{section}
\section{Introduction}\label{S:Introduction}
Fix a field $k$ and an algebraic group $G$ over $k$. Ben-Zvi posed the
following question~\cite{165701}: if $k$ has positive characteristic, then is
the unbounded derived category of representations of $G$ compactly generated?

The second author recently answered Ben-Zvi's question negatively in
the case of $\Ga$~\cite[Rem.~4.2]{MR3212862}. We establish a much stronger version
of this result: in the unbounded derived category
of representations of $\Ga$ in positive characteristic, there are no compact
objects besides $0$ (Proposition \ref{P:BGa:nocompacts}).

We say that $G$ is \emph{\badgroup} if $k$ has positive characteristic and
$\overline{G}=G\tensor_k \overline{k}$ has a subgroup isomorphic to $\Ga$, or,
equivalently, if $\red{\bar{G}}^0$ is not semi-abelian (Lemma \ref{L:badgroup}). Examples of {\badgroup}
groups are $\Ga$ and $\GL_n$. The results of this article imply that in
positive characteristic, the derived category of representations of $G$ is not
compactly generated if $G$ is {\badgroup}. Conversely, when $G$ is not
{\badgroup}, the first and third author show that its derived category of representations is compactly generated \cite[Thm.~A]{hallj_dary_alg_groups_classifying}.  Ben-Zvi's question is thus completely resolved.

A somewhat subtle point that we have suppressed so far is that there are two
potential ways to look at the unbounded derived category of representations
of~$G$. First, there is $\DCAT(\REP(G))$; second, there is $\DQCOH(BG)$, the
unbounded derived category of lisse-\'etale $\Orb_{BG}$-modules with
quasi-coherent cohomology. There is a natural functor $\DCAT(\REP(G)) \to
\DQCOH(BG)$ and if $G$ is affine, then this functor induces an equivalence on
bounded below derived categories.


In the present article, we will show that in positive characteristic if $G$ is
affine and {\badgroup}, then this
functor is not full. We also prove that if $G$ is {\badgroup},
then neither $\DCAT(\REP(G))$ nor $\DQCOH(BG)$ is compactly generated.

The results above are actually special cases of some general results for
unbounded derived categories of quasi-coherent sheaves on algebraic stacks. We
say that an algebraic stack is \emph{\badstack} (see \S\ref{S:gen_case}) if it
has a point with {\badgroup} stabilizer group.
Our first main result is the following.
\begin{theorem}\label{T:compact-generation}
  Let $X$ be an algebraic stack that is quasi-compact, 
  quasi-separated and \badstack.
  \begin{enumerate}
  \item The triangulated category $\DQCOH(X)$ is not compactly
    generated.
  \item If $X$ has affine diagonal or is noetherian and affine-pointed, then
    $\DCAT(\QCOH(X))$ is not compactly generated. 
  \end{enumerate}
\end{theorem}
An algebraic stack $X$ is \fndefn{affine-pointed} if every morphism $\spec k \to X$, where $k$ is a 
field, is affine. If $X$ has quasi-affine or quasi-finite diagonal, then $X$ is affine-pointed \cite[Lem.~4.5]{hallj_dary_coherent_tannakian_duality}.

We wish to point out that Theorem \ref{T:compact-generation} is counter to the prevailing wisdom. Indeed, let $X$ be a quasi-compact and quasi-separated algebraic stack. If $X$ is
a scheme, then it is well-known that $\DQCOH(X)$ is compactly generated by perfect complexes
\cite[Thm.~3.1.1(b)]{MR1996800}. More generally, recent work of Krishna \cite[Lem.~4.8]{MR2570954},
Ben-Zvi--Francis--Nadler \cite[\S3.3]{MR2669705}, To\"en \cite[Cor.~5.2]{MR2957304},
and the first and third authors \cite{perfect_complexes_stacks}, has
shown
that the unbounded derived category $\DQCOH(X)$ is
compactly generated by perfect complexes if $X$ is a Deligne--Mumford stack or a stack in characteristic zero that \'etale-locally has the resolution property \cite{MR2108211,2013arXiv1306.5418G}.

Also recall that if $X$ is a scheme that is either
quasi-compact with affine diagonal or noetherian, then the functor
$\Psi_X\colon \DCAT(\QCOH(X)) \to \DQCOH(X)$ is an equivalence of
triangulated categories---see \cite[Cor.~5.5]{MR1214458} for the
separated case (the argument adapts trivially to the case of affine
diagonal) and
\cite[Tags \spref{08H1} \& \spref{09TN}]{stacks-project}
in the setting of algebraic spaces. Our second main result is a partial extension of this to algebraic stacks.
\begin{theorem}\label{T:cpt-gen_Psi-equiv}
Let $X$ be an algebraic stack that is quasi-compact with affine
diagonal or noetherian and affine-pointed. If $\DQCOH(X)$ is compactly generated, then the
functor $\Psi_X\colon \DCAT(\QCOH(X))\to \DQCOH(X)$ is an
equivalence of categories.
\end{theorem}
In particular, $\Psi_X$ is an equivalence for every Deligne--Mumford
stack with affine diagonal, every noetherian Deligne--Mumford stack, and every stack in characteristic zero with affine diagonal that \'etale-locally has the resolution
property~\cite{perfect_complexes_stacks}. This is a vast extension of 
work of Krishna~\cite[Cor.~3.7]{MR2570954}, where $\Psi_X$ was proved to be an
equivalence for every Deligne--Mumford stack that is separated, of finite
type over a field of characteristic $0$, has the resolution property and
whose coarse moduli space is a scheme.

It is natural to ask whether
$\Psi_X$ is always an equivalence of categories. On the positive side, we
prove that the restricted functor $\Psi^+_X\colon \DCAT^+(\QCOH(X)) \to
\DQCOH^+(X)$ is an equivalence of triangulated categories when either $X$ is
quasi-compact with affine diagonal or noetherian and affine-pointed (Theorem
\ref{T:bdd_below}, also see \cite[Thm.~3.8]{lurie_tannaka} and
\cite[Prop.~II.3.5]{MR0354655}). On the negative side, we have the following
result.

\begin{theorem}\label{T:notfull}
  Let $X$ be an algebraic stack that is quasi-compact with affine
  diagonal or noetherian and affine-pointed. If $X$ is \badstack, then the functor
  $\Psi_X\colon \DCAT(\QCOH(X)) \to \DQCOH(X)$ is not full.
\end{theorem}

We were unable to determine whether the functor $\Psi_X$ in Theorem 
\ref{T:notfull} is faithful or not. For stacks with non-affine stabilizer groups the situation is even worse:
if $X=BE$, where $E$ is an elliptic
curve over $\mathbb{C}$, then the functor $\Psi_X^b\colon \DCAT^b(\COH(X))
\to\DCAT^b_{\mathrm{Coh}}(X)$ is neither essentially surjective nor
full. 

Note that when $X$ has affine
diagonal or is noetherian and affine-pointed, the first claim in Theorem \ref{T:compact-generation} is a trivial consequence
of Theorems \ref{T:cpt-gen_Psi-equiv} and \ref{T:notfull}. 
\subsection*{Left-completeness}
In the course of proving Theorem \ref{T:notfull}, we will prove
that the triangulated category $\DCAT(\QCOH(X))$ is not left-complete
whenever $X$ is {\badstack} with affine diagonal. This
generalizes an example of Neeman \cite{MR2875857} and amplifies some
observations of Drinfeld--Gaitsgory
\cite[Rem.~1.2.10]{MR3037900}.

In Appendix 
\ref{APP:left-complete}, we will prove that $\DQCOH(X)$ is left-complete
for all algebraic stacks $X$. An analogous assertion in the context of
derived algebraic geometry has been addressed by Drinfeld--Gaitsgory
\cite[Lem.~1.2.8]{MR3037900}. In the Stacks Project
\cite[\href{http://stacks.math.columbia.edu/tag/08IY}{08IY}]{stacks-project}
a similar result has been proved, albeit in a different context. 

A requirement for a triangulated category to be left-complete is that
it admits countable products. We were unable to locate a proof in the
literature that $\DQCOH(X)$ admits countable products, however. Thus we
also address this in Appendix \ref{APP:left-complete}. By
\cite[Cor.~1.18]{MR1812507}, it suffices to prove that $\DQCOH(X)$
is well generated. 
\subsection*{Well generation}
If $X$ is an algebraic stack, then it is known that $\DQCOH(X)$ is well generated in the following cases:
\begin{itemize}
\item $X$ is Deligne--Mumford \cite[Prop.~2.3.13(3)]{dag8}; or
\item $X$ is quasi-compact with affine diagonal \cite[Prop.~3.4.17(1)]{dag8}; or
\item $X$ is a $\Q$-stack that is 
  quasi-compact and quasi-separated with affine stabilizers and finitely presented inertia 
  \cite[Thm.~1.4.10]{MR3037900}. 
\end{itemize}
This list is quite extensive, covering the majority of algebraic stacks encountered in this article and in practice. We are not aware of a reference that $\DQCOH(X)$ is well generated in complete generality, however.

In
Appendix \ref{APP:local-grothendieck} we show that if $\ca$ is a Grothendieck
abelian category and $\cm \subseteq
\ca$ is a weak Serre subcategory that is closed under coproducts and is
Grothendieck abelian, then
$\DCAT_{\cm}(\ca)$ is a well generated triangulated category---a result we expect to be of
independent interest.
Since the inclusion $\QCOH(X) \subseteq \MOD(X)$ has these properties, this establishes that $\DQCOH(X)$ is well generated.

While we hoped that the well generation of $\DCAT_{\cm}(\ca)$, which is essentially a 
set-theoretic issue, would follow from some general results in the theory of presentable categories (even of the 
$\infty$-category type), we did not succeed along those lines. Instead, we give a direct proof using the Gabriel--Popescu Theorem.
We also wish to point out that 
while \cite[Prop.~14.2.4]{MR2182076} is quite general, it does not apply in
our situation. Indeed, they
require that the embedding $\cm \subseteq \ca$ is closed under
$\ca$-subquotients (i.e., $\cm$ is a Serre subcategory of $\ca$), which is not the case for 
$\QCOH(X) \subseteq \MOD(X)$.  

Note that we do not understand, in general, under what circumstances
 localizing subcategories of well generated categories
are generated by sets of objects. The question goes back to the early 
1970s, when Adams claimed to prove that it is possible to localize spectra
with respect to arbitrary homology and cohomology theories. Bousfield pointed
out the gap in Adams' argument and half-fixed it: he showed that it
is possible to localize spectra with respect to homology, see
\cite{Bousfield75,Bousfield79A}.

We restate the problem in the language here: let $\ct$ be the homotopy category
of spectra, which is well-known to be well generated (even compactly
generated). Given a homology theory $E$ and a cohomology theory $F$, we
can define localizing subcategories
\[
\mathcal{L}(E)=\{X\in\ct\suchthat E(X)=0\},\qquad
\mathcal{L}(F)=\{X\in\ct\suchthat F(X)=0\}.
\]
Bousfield proved that $\mathcal{L}(E)$ has a set of generators, and conjectured
that so does $\mathcal{L}(F)$. This conjecture remains open, 
despite considerable
effort over the intervening four decades. Some recent highlights of the literature are 
\cite{Casacuberta-Scevenels-Smith05,MR1997044,MR2166357}.

\section{Preliminaries}
Let $\phi\colon X \to Y$ be a quasi-compact and quasi-separated
morphism of algebraic stacks. Then the restriction of the functor $\MODPSH{\phi} \colon \MOD(X) 
\to \MOD(Y)$ to $\QCOH(X)$ factors through $\QCOH(Y)$ \cite[Lem.~6.5(i)]{MR2312554}, giving rise to a functor $\QCOHPSH{\phi} \colon \QCOH(X) \to \QCOH(Y)$. Since the categories $\MOD(X)$ and $\QCOH(X)$ are
Grothen\-dieck abelian
\cite[\href{http://stacks.math.columbia.edu/tag/0781}{0781}]{stacks-project},  the unbounded derived functors of $\MODPSH{\phi}$ and $\QCOHPSH{\phi}$ exist
\cite[\href{http://stacks.math.columbia.edu/tag/079P}{079P} \& \href{http://stacks.math.columbia.edu/tag/070K}{070K}]{stacks-project},
and we denote these as $\RDERF\MODPSH{\phi}$ and
$\RDERF \QCOHPSH{\phi}$, respectively. By \cite[Lem.~6.20]{MR2312554}, the
restriction of $\RDERF\MODPSH{\phi}$ to $\DQCOH^+(X)$ factors uniquely
through $\DQCOH^+(Y)$. If, in addition, $\phi$ is concentrated (e.g.,
representable), then the restriction of $\RDERF\MODPSH{\phi}$ to
$\DQCOH(X)$ factors through $\DQCOH(Y)$ (see \cite[Lem.~2.1]{hallj_coho_bc} for the representable case and \cite[Thm.~2.6(2)]{perfect_complexes_stacks} in general).

For an algebraic stack $W$ let $\Psi_W \colon \DCAT(\QCOH(W)) \to
\DQCOH(W)$ denote the natural functor. The universal properties of
right-derived functors provide a diagram:
\[
\xymatrix@C10ex{\DCAT(\QCOH(X))
  \ar[r]^-{\RDERF\QCOHPSH{\phi}} \ar[d]_{\Psi_X} & \ar[d]^{\Psi_Y} \DCAT(\QCOH(Y))
  \\ \DQCOH(X)\ar[r]^-{\RDERF\MODPSH{\phi}} & \DQCOH(Y), }
\] 
together with a natural transformation of functors:
\begin{equation}
  \epsilon_\phi \colon \Psi_Y\circ \RDERF\QCOHPSH{\phi}
  \Rightarrow \RDERF\MODPSH{\phi} \circ \Psi_X.\label{eq:natural_trans}
\end{equation}
The following result, for schemes, is well-known
\cite[B.8]{MR1106918}; for algebraic spaces, see \cite[Tags \spref{09TH} \& \spref{08GX}]{stacks-project}.
\begin{proposition}\label{P:coho_all}
  Let $\phi\colon X \to Y$ be a morphism of algebraic stacks. Suppose
  that both $X$ and $Y$ are quasi-compact with affine diagonal or noetherian and 
  affine-pointed. If
  $M \in \DCAT^+(\QCOH(X))$, then the morphism induced by (\ref{eq:natural_trans}):
  \[
  \epsilon_\phi(M) \colon \Psi_Y\circ \RDERF\QCOHPSH{\phi}(M)
  \to \RDERF\MODPSH{\phi} \circ \Psi_X(M)
  \]
  is an isomorphism. In particular, since $\Psi_Y^+ \colon
  \DCAT^+(\QCOH(Y)) \to \DQCOH^+(Y)$ is an equivalence
  (Theorem \ref{T:bdd_below}), it follows that there is a natural
  isomorphism for each $M \in \DCAT^+(\QCOH(X))$:
  \[
  \RDERF\QCOHPSH{\phi}(M) \to (\Psi^+_Y)^{-1}\circ \RDERF\MODPSH{\phi} \circ \Psi_X^+(M).
  \]
\end{proposition}
\begin{proof}
  The functors $\QCOHPSH{\phi}$ and $\MODPSH{\phi}$ are left-exact, thus the
  functors $\RDERF\QCOHPSH{\phi}$ and $\RDERF\MODPSH{\phi}$ are
  bounded below. Via standard ``way-out'' arguments, one readily
  reduces to proving the isomorphism above in the case
  $M\simeq N[0]$, where $N \in \QCOH(X)$. The isomorphism, in this case,
  reduces to proving that if $N\in \QCOH(X)$, then the natural
  morphism $\RDERF^i\QCOHPSH{\phi}N \to \RDERF^i\MODPSH{\phi}N$ is an
  isomorphism for all integers $i\geq 0$, where
  $\RDERF^i\QCOHPSH{\phi}$ (resp.~$\RDERF^i\MODPSH{\phi}$) denotes
  the $i$th right-derived functor of $\QCOHPSH{\phi}$
  (resp.~$\MODPSH{\phi}$). A standard $\delta$-functor argument shows that it is 
  sufficient to prove that $\RDERF^i(\phi_{\lisset})_*I = 0$ for every $i>0$ and injective $I$ 
  of $\QCOH(X)$. This now follows from Lemma 
  \ref{L:bdd_below_coho}\itemref{L:bdd_below_coho:local}.
\end{proof}
\begin{corollary}\label{C:dsum}
  Let $\phi\colon X \to Y$ be a concentrated morphism of algebraic stacks.
  If $X$ and $Y$ are quasi-compact with affine
  diagonal or noetherian and affine-pointed, then there exists an integer $r\geq 0$ such that 
  for all $M \in \DCAT(\QCOH(X))$ and integers $n$ the natural map:
  \[
  \tau^{\geq n}\RDERF\QCOHPSH{\phi}M \to \tau^{\geq
    n}\RDERF\QCOHPSH{\phi}\tau^{\geq n-r}M
  \]
  is a quasi-isomorphism. It follows that
  \begin{enumerate}
  \item\label{CI:dsum:coprod}
    $\RDERF\QCOHPSH{\phi}$ preserves small coproducts,
  \item for all $M\in \DCAT(\QCOH(X))$ the natural morphism induced by
    (\ref{eq:natural_trans}):
    \[
    \epsilon_\phi(M) \colon \Psi_Y\circ \RDERF\QCOHPSH{\phi}M
    \to \RDERF\MODPSH{\phi} \circ \Psi_X(M)
    \]
    is an isomorphism, and
  \item\label{CI:dsum:qaff-cons}
    if $\phi$ is quasi-affine, then $\RDERF\QCOHPSH{\phi}$ is
    conservative.
  \end{enumerate}
\end{corollary}
\begin{proof}
  The claims (1)--(3) are all simple consequences of the main claim
  and Proposition \ref{P:coho_all}. Since $\phi$ is a concentrated
  morphism and $Y$ is quasi-compact and quasi-separated, there exists an integer $r\geq 0$ such that if $N \in
  \QCOH(X)$, then $\RDERF^i\MODPSH{\phi}N = 0$ for all $i>r$. By
  Proposition \ref{P:coho_all} it follows that
  $\RDERF^i\QCOHPSH{\phi}N = 0$ for all $i>r$ too. The result now
  follows from
  \cite[\href{http://stacks.math.columbia.edu/tag/07K7}{07K7}]{stacks-project}.
\end{proof}
\begin{corollary}\label{C:BGa:compacts_image_coherent}
  Let $X$ be an algebraic stack that is quasi-compact with affine diagonal or noetherian and 
  affine-pointed. If
  $C$ is a compact object of either $\DCAT(\QCOH(X))$ or $\DQCOH(X)$,
  then $C$ is perfect. Moreover if $X$ is noetherian, then $C$ is
  quasi-isomorphic to a bounded complex of coherent sheaves on $X$.
\end{corollary}
\begin{proof}
  Let $C$ be a compact object of $\DQCOH(X)$. By
  \cite[Lem.~4.4(1)]{perfect_complexes_stacks}, $C$ is a perfect complex
  and in particular belongs to $\DQCOH^b(X) \subseteq \DQCOH^+(X)$. By
  Theorem~\ref{T:bdd_below}, it follows that $C\simeq
  \Psi_X(\tilde{C})$ for some $\tilde{C} \in\DCAT(\QCOH(X))$. If $X$
  is noetherian, $\tilde{C}$ even belongs to
  $\DCAT_{\COH(X)}^b(\QCOH(X))$. Combining
  \cite[Prop.~15.4]{MR1771927} with \cite[II.2.2]{MR0354655}, we
  deduce that $C$ belongs to the image of $\DCAT(\COH(X)) \to
  \DQCOH(X)$.

  Now let $C$ be a compact object of $\DCAT(\QCOH(X))$. Let $p\colon U
  \to X$ be a smooth surjection from an affine scheme $U$. The functor
  $p^*_{\QCOH} \colon \QCOH(X) \to \QCOH(U)$ is exact and gives rise to a
  derived functor $\LDERF p_{\QCOH}^*\colon \DCAT(\QCOH(X)) \to
  \DCAT(\QCOH(U))$. Its right adjoint $\RDERF
  \QCOHPSH{p}$ preserves small coproducts by
  Corollary \ref{C:dsum}\eqref{CI:dsum:coprod}, so $\LDERF p_{\QCOH}^*C
  \in \DCAT(\QCOH(U))$ is compact. Since $U=\spec A$ is
  affine, it follows that $\QCOH(U) \cong \MOD(A)$ and so $\LDERF
  p_{\QCOH}^*C$ is a perfect complex
  \cite[\href{http://stacks.math.columbia.edu/tag/07LT}{07LT}]{stacks-project}.
  If $X$ is noetherian, then $C \in \DCAT^b_{\COH(X)}(\QCOH(X))$. Arguing as
  before, we deduce that $C$ belongs to the image of $\DCAT(\COH(X))
  \to \DCAT(\QCOH(X))$.
\end{proof}
In the following Lemma we will give a sufficient condition for compactness
of a perfect object in $\DCAT(\QCOH(X))$. We do not know if this condition
is necessary. The analogous condition in $\DQCOH(X)$ is necessary~\cite[Lem.~4.5]{perfect_complexes_stacks}.
\begin{lemma}\label{L:almost_char_compact}
  Let $X$ be an algebraic stack that is quasi-compact with affine diagonal or noetherian and 
  affine-pointed. Let $P\in \DCAT(\QCOH(X))$ be a perfect complex. Consider the following
  conditions
  \begin{enumerate}
  \item\label{item:almost_char_compact:base} $P$ is a compact
    object of $\DCAT(\QCOH(X))$.
  \item\label{item:almost_char_compact:bddext} There exists an integer
    $r\geq 0$ such that $\Hom_{\Orb_X}(P,N[i])=0$ for all
    $N\in\QCOH(X)$ and $i>r$.
  \item\label{item:almost_char_compact:stabletrunc} There exists an
    integer $r\geq 0$ such that the natural map
    \[
    \tau^{\geq j}\RHom_{\DCAT(\QCOH(X))}(P,M) \to \tau^{\geq
      j}\RHom_{\DCAT(\QCOH(X))}(P,\tau^{\geq j-r}M) 
    \]
    is a quasi-isomorphism for all $M \in \DCAT(\QCOH(X))$ and
    integers $j$.
  \end{enumerate}
  Then \itemref{item:almost_char_compact:bddext} and
  \itemref{item:almost_char_compact:stabletrunc} are equivalent and imply
  \itemref{item:almost_char_compact:base}.
\end{lemma}
\begin{proof}
  Condition \itemref{item:almost_char_compact:bddext} is a special case
  of \itemref{item:almost_char_compact:stabletrunc}: let $M=N[i]$
  and $j=0$.

  Conversely, assume that condition \itemref{item:almost_char_compact:bddext}
  holds and let $M \in \DCAT(\QCOH(X))$.
  Since $\QCOH(X)$ is Grothendieck abelian, there is a
  quasi-isomorphism $M \to I^\bullet$ in $\DCAT(\QCOH(X))$, where
  $I^\bullet$ is K-injective and $I^j$ is injective for every integer
  $j$ \cite{MR1948842}.

  Let $p\geq r+1$ be an integer with the property that
  $P \in \DCAT^{\geq -p+1}(\QCOH(X))$. Then the natural morphism of
  chain complexes:
  \begin{equation}
    \tau^{\geq j}\Hom^\bullet_{\KCAT(\QCOH(X))}(P,I^\bullet) \to
    \tau^{\geq j}\Hom^{\bullet}_{\KCAT(\QCOH(X))}(P,\sigma^{\geq
      j-p}I^{\bullet}), \label{eq:brutal}
  \end{equation}
  where $\sigma$ is the brutal truncation, is a
  quasi-isomorphism. For every integer $j$ there is also a
  morphism $s_j\colon \sigma^{\geq j}I^\bullet \to \trunc{\geq
    j}M$. If $C_j^\bullet$ is the mapping cone of $s_j$, then
  $C_j^\bullet \homotopic d(I^{j-1})[-(j-1)]$.
  Thus, by condition \itemref{item:almost_char_compact:bddext}, we have
  for every integer $j$
  \[
  \trunc{\geq j+r}\RHom_{\DCAT(\QCOH(X))}(P,C_j^\bullet) \homotopic 0.
  \]
  Since there is also
  distinguished triangle for every integer $j$:
  \[
  \RHom_{\Orb_X}(P,\sigma^{\geq j-p}I^\bullet) \to
  \RHom_{\Orb_X}(P,\trunc{\geq j-p}M) \to
  \RHom_{\Orb_X}(P,C_{j-p}^\bullet),
  \]
  it follows that for every integer $j$ there is a quasi-isomorphism:
  \[
  \trunc{\geq j}  \RHom_{\Orb_X}(P,\sigma^{\geq
    j-p}I^\bullet) \homotopic \trunc{\geq
    j}\RHom_{\Orb_X}(P,\trunc{\geq j-p}M). 
  \]
  For every integer $j$, we also have a distinguished triangle
  \[
  H^{j-r-1}(M)[-(j-r-1)]\to \tau^{\geq j-r-1}M\to \tau^{\geq j-r}M.
  \]
  As before, it follows that $\tau^{\geq j}\RHom_{\Orb_X}(P,\tau^{\geq
    j-r-1}M)\homotopic \tau^{\geq j}\RHom_{\Orb_X}(P,\tau^{\geq j-r}M)$ and
  thus by induction:
  \[
  \tau^{\geq j}\RHom_{\Orb_X}(P,\tau^{\geq j-p}M) \to \tau^{\geq
    j}\RHom_{\Orb_X}(P,\tau^{\geq j-r}M).
  \]
  Combining these quasi-isomorphisms with 
  \eqref{eq:brutal} gives \itemref{item:almost_char_compact:stabletrunc}.

  For \itemref{item:almost_char_compact:stabletrunc} implies
  \itemref{item:almost_char_compact:base}: this follows
  from Theorem \ref{T:bdd_below}
  and~\cite[Lem.~1.2(3)]{perfect_complexes_stacks}.
\end{proof}
We now relate compact generation in $\DCAT(\QCOH(X))$ with
compact generation in $\DQCOH(X)$.
\begin{lemma}\label{L:compacts-aff-diag}
  Let $X$ be an algebraic stack that is quasi-compact with affine diagonal or noetherian and 
  affine-pointed.
  \begin{enumerate}
  \item\label{L:compacts-aff-diag:item:char} If $P\in \DCAT(\QCOH(X))$
    is a perfect complex such that $\Psi(P)$ is compact
    in $\DQCOH(X)$, then $P$ is compact in $\DCAT(\QCOH(X))$.
  \item\label{L:compacts-aff-diag:item:fcd} If $X$ has finite
    cohomological dimension, then every perfect complex is compact in
    both $\DCAT(\QCOH(X))$ and $\DQCOH(X)$.
  \item\label{L:compacts-aff-diag:item:gens} If a set of objects
    $\{P_i\}$ of $\DCAT(\QCOH(X))$ has the property that
    $\{\Psi(P_i)\}$ compactly generates $\DQCOH(X)$, then $\{P_i\}$
    compactly generates $\DCAT(\QCOH(X))$.
  \end{enumerate}
\end{lemma}
\begin{proof}
  For \itemref{L:compacts-aff-diag:item:char}, by
  \cite[Lem.~4.5]{perfect_complexes_stacks}, since $\Psi(P)$ is
  compact, there exists an integer $r$ such that if $i>r$ and $N \in
  \QCOH(X)$, then $\Hom_{\Orb_X}(\Psi(P),N[i]) = 0$. The functor
  $\Psi^+$ is an equvialence (Theorem \ref{T:bdd_below}), so
  $\Hom_{\Orb_X}(P,N[i])=0$ for all $i>r$ and $N\in \QCOH(X)$.
  It follows that $P$ is compact by
  Lemma~\ref{L:almost_char_compact}.

  Statement \itemref{L:compacts-aff-diag:item:fcd} is a direct consequence
  of \itemref{L:compacts-aff-diag:item:char} and
  \cite[Lem.~4.4(3)]{perfect_complexes_stacks}.

  For \itemref{L:compacts-aff-diag:item:gens}, let
  $M \in \DCAT(\QCOH(X))$. If $P$ is perfect and $\Psi(P)$ is compact, then
  $\RHom(P,M)=\RHom(\Psi(P),\Psi(M))$.
  Indeed, there exists an integer $r$ such that for all integers $j$
  \begin{align*}
  \tau^{\geq j}\RHom(P,M) &\homotopic
  \tau^{\geq j}\RHom(P,\tau^{\geq j-r}M) \\
  &\homotopic
  \tau^{\geq j}\RHom(\Psi(P),\tau^{\geq j-r}\Psi(M))\homotopic
  \tau^{\geq j}\RHom(\Psi(P),\Psi(M)),
  \end{align*}
  by Lemma~\ref{L:almost_char_compact} and
  \cite[Lem.~4.5]{perfect_complexes_stacks} since $\Psi^+$ is an
  equivalence of triangulated categories~(Theorem \ref{T:bdd_below})
  and $\Psi$ is t-exact.
  Thus, if $\Hom_{\DCAT(\QCOH(X))}(P_i[l],M) = 0$ for all $i$ and integers $l$,
  then $\Hom_{\Orb_X}(\Psi(P_i)[l],\Psi(M)) = 0$ for all $i$ and $l$. It follows
  that $\Psi(M)=0$ and, since $\Psi$ is conservative, that $M=0$.
\end{proof}
The following lemma, while technical, gives an explicit description of an adjunction that is useful in the article.
\begin{lemma}\label{L:adjunction}
  Let $X$ be an algebraic stack and let $M \in \DCAT(\QCOH(X))$.
  \begin{enumerate}
  \item\label{L:adjunction:exists} The functor $\Psi_X \colon \DCAT(\QCOH(X)) \to 
    \DQCOH(X)$ admits a right adjoint $\Phi_X \colon \DQCOH(X) \to \DCAT(\QCOH(X))$. 
  \item\label{L:adjunction:aff_prop} If $X$ is a 
    quasi-compact with affine diagonal or noetherian and affine-pointed, then there exists a 
    compatible quasi-isomorphism:
    \[
    \Phi_X\Psi_X(M) \simeq \holim{n}\trunc{\geq -n}M.
    \]
  \end{enumerate}
\end{lemma}
\begin{proof}
  We suppress the subscript $X$ from $\Psi$ and $\Phi$ throughout. Since $\Psi$ preserves small coproducts and $\DCAT(\QCOH(X))$ is well generated 
  \cite[Thm.\ 0.2]{MR1874232}, $\Psi$ admits a right adjoint $\Phi \colon \DQCOH(X) 
  \to \DCAT(\QCOH(X))$ \cite[Prop.\ 1.20]{MR1812507}. This proves 
  \itemref{L:adjunction:exists}.
  
  To prove \itemref{L:adjunction:aff_prop}, by left-completeness of 
  $\DQCOH(X)$ (Theorem \ref{APP:T:left-complete}), 
  \[
  \Phi\Psi(M) \to \Phi(\holim{n}\trunc{\geq -n}\Psi(M))
  \]
  is a quasi-isomorphism. Since $\Phi$ is a right adjoint, it preserves homotopy limits. Also, 
  $\Psi$ is t-exact. Hence, there is a quasi-isomorphism
  \[
  \Phi(\holim{n}\trunc{\geq -n}\Psi(M)) \simeq \holim{n} \Phi\Psi(\trunc{\geq -n}M).
  \]
  By Theorem \ref{T:bdd_below}, however, $\trunc{\geq -n}M \simeq \Phi\Psi(\trunc{\geq 
    -n}M)$. This proves the claim.
\end{proof}
\begin{remark}\label{R:lc_ff}
  From Lemma \ref{L:adjunction}\itemref{L:adjunction:aff_prop} it is immediate that when 
  $X$ is quasi-compact with affine diagonal or 
  noetherian and affine-pointed, the left-completeness of $\DCAT(\QCOH(X))$ is equivalent 
  to $\Psi_X$ being fully faithful. 
\end{remark}
We now prove Theorem \ref{T:cpt-gen_Psi-equiv} using an argument similar to \cite[Lem.~4.5]{MR2846489}.
\begin{proof}[Proof of Theorem~\ref{T:cpt-gen_Psi-equiv}]
By Lemma~\ref{L:compacts-aff-diag}\itemref{L:compacts-aff-diag:item:gens},
both
$\DCAT(\QCOH(X))$ and $\DQCOH(X)$ are compactly generated and $\Psi$
takes a set of compact generators to a set of compact generators. In
particular, the right adjoint $\Phi \colon \DQCOH(X)\to
\DCAT(\QCOH(X))$ of $\Psi$ preserves small
coproducts~\cite[Thm.~5.1]{MR1308405}.

Consider the unit $\eta_M\colon M\to \Phi\Psi(M))$ and the counit
$\epsilon_M\colon \Psi\Phi(M)\to M$ of the adjunction.  Since $\Psi^+$
is an equivalence, we have that $\eta_P$ and $\epsilon_P$ are isomorphisms for
every compact object $P$. Since $\eta$ and $\epsilon$ are triangulated functors
that preserve small coproducts and $\DQCOH(X)$ and $\DCAT(\QCOH(X))$ are compactly
generated, it follows that $\eta$ and $\epsilon$ are equivalences. We conclude
that $\Psi$ is an equivalence.
\end{proof}

\section{The case of \texorpdfstring{$B_k\Ga$}{BGa} in positive characteristic}\label{S:BGa}
Throughout this section we let $k$ denote a field of
characteristic $p>0$. Let $B_k\Ga$ be the algebraic stack
classifying $\Ga$-torsors over $k$.  We remind ourselves that the
category of quasi-coherent sheaves on $B_k\Ga$ is the category of
$\Ga$-modules, which is equivalent to the category of locally small
modules over a certain ring $R$. In fact $R$ is the ring
\[
R=\frac{k[x_1^{},x_2^{},x_3^{},\ldots]}{(x_1^{p},x_2^{p},x_3^{p},\ldots)}
\]
and a module is locally small if every element is annihilated by all
but
finitely many $x_i$. Let us write $\DCAT(R^{\ls})$ for the derived category
of the category of locally small $R$-modules, and observe that
$\DCAT(R^{\ls})\cong\DCAT(\QCOH(B_k\Ga))$.
\begin{proposition}\label{P:BGa:nocompacts}
  The only compact objects, in either $\DCAT(\QCOH(B_k\Ga))$ or
  $\DQCOH(B_k\Ga)$, are the zero objects.
\end{proposition}
\begin{proof} The algebraic stack $B_k\Ga$ is noetherian with affine diagonal
and so, by Corollary
\ref{C:BGa:compacts_image_coherent}, every compact object is the image
of a bounded complex of coherent sheaves. Let $C$ be a compact object;
we need to show that $C$ vanishes.

 Our compact
object 
$C$ is the image of a finite complex of finitely generated modules
in $\DCAT(R^{\ls})$. In particular,
there exists an integer $n>1$
such that $x_i^{}$ annihilates $C$ for all $i\geq
n$. Let us put this slightly differently: consider the ring homomorphisms
$S\stackrel\alpha \to T\stackrel\beta \to R\stackrel\gamma \to T$ where 
\[
S=k[x_n^{}]/(x_n^p),\qquad T=\frac{k[x_1^{},x_2^{},\ldots,x_{n-1}^{},x_{n}^{}]}{(x_1^{p},x_2^{p},\ldots,x_{n-1}^{p},x_{n}^{p})}
\]
where the maps $S\stackrel\alpha \to T\stackrel\beta \to R$ are the natural 
inclusions, and where 
$\gamma\colon R \to T$ is defined by
\[
\gamma(x_i)=\left\{
\begin{array}{lcl}
x_i&\quad&\text{if }i\leq n\\
0&\quad&\text{if }i> n.
\end{array}
\right.
\]  
Note that $\gamma\beta=\id$.
Restriction of scalars gives induced maps of derived categories, which we
write as $\DCAT(T)\xrightarrow{\gamma_*} 
\DCAT(R^{\ls})\xrightarrow{\beta_*} \DCAT(T)
\xrightarrow{\alpha_*}\DCAT(S)$, 
and $\beta_*^{}\gamma_*^{}=\id$. Our complex $C$, which is a bounded complex
annihilated by $x_i$ for all $i\geq n$, is of the form $\gamma_*^{}B$ where
$B\in\DCAT^b(T)$ is a bounded complex of finite $T$-modules. And the 
fact that $x_n$ annihilates $C$ translates to saying that 
 $\alpha_*^{}B$ is a complex of modules annihilated by $x_n$, that is a
complex of $k$-vector spaces. We wish to show that $C=0$ or, equivalently,
that $\alpha_*^{}B$ is acyclic. We will show that if $C$ is non-zero, then
this gives rise to a contradiction.

Thus, assume that the
cohomology of $\alpha_*^{}B$ is non-trivial: in $\DCAT(S)$ the complex 
$\alpha_*^{}B$ is
isomorphic
to a non-zero sum of suspensions $k[\ell]$ of $k$.
Then there are infinitely many integers $m$
and non-zero maps in $\DCAT(S)$ of the form $\alpha_*^{}B \to k[m]$. Indeed,
$\Ext^m_S(k,k)\neq0$ for all $m\geq0$. But $\alpha_*^{}$ has a right
adjoint $\alpha^\times=\RHom_S(T,-)$, and we deduce 
infinitely many non-zero maps in $\DCAT(T)$ of the form
$B \to \alpha^\times k[m]=\Hom_S^{}(T,k)[m]$. Since $\DCAT(T)$ is
left-complete, these combine to a map in $\DCAT(T)$
\[
B \xrightarrow{\Psi} \prod_{m}\Hom_S^{}(T,k)[m]\cong \coprod_{m}\Hom_S^{}(T,k)[m]
\]
for which the composites
\[
B \xrightarrow{\Psi} \coprod_{m}\Hom_S^{}(T,k)[m] \xrightarrow{\pi_m} \Hom_S^{}(T,k)[m]
\]
are non-zero.
Applying $\gamma_*^{}$, which preserves coproducts, we deduce maps
\[
\gamma_*^{}B \xrightarrow{\gamma_*\Psi}
\coprod_{m}\gamma_*^{}\Hom_S^{}(T,k)[m] \xrightarrow{\gamma_*\pi_m}  
\gamma_*^{}\Hom_S^{}(T,k)[m]
\]
whose composites cannot vanish in $\DCAT(R^{\ls})$, since $\beta_*^{}$
takes them to non-zero maps. The equivalence
$\DCAT(R^{\ls})\cong\DCAT(\QCOH(B_k\Ga))$ gives us that the composites
in $\DCAT(\QCOH(B_k\Ga))$ do not vanish. Furthermore, the composites lie in
$\DCAT^+(\QCOH(B_k\Ga))\subseteq\DCAT(\QCOH(B_k\Ga))$, and on
$\DCAT^+(\QCOH(B_k\Ga))$ the map to $\DQCOH(B_k\Ga)$ is fully faithful
\cite[Thm.~3.8]{lurie_tannaka}. Hence the images of the composites are
non-zero in $\DQCOH(B_k\Ga)$ as well.  But this contradicts the
compactness of $C=\gamma_*^{}B$. 
\end{proof}
\section{The general case}\label{S:gen_case}
In this section we extend the results of the previous section and show that the
presence of $\Ga$ in the stabilizer groups of an algebraic stack $X$ is an
obstruction to compact generation in positive characteristic. The existence of
finite unipotent subgroups such as $\zz/p\zz$ and $\Galpha_p$ is an obstruction
to the compactness of the structure sheaf $\mathcal{O}_X$ but does not rule out compact
generation \cite{hallj_dary_alg_groups_classifying}. The only connected groups in characteristic $p$ without unipotent subgroups are the groups of
multiplicative type. The following well-known lemma characterizes the groups
without $\Ga$'s.
\begin{lemma}
Let $G$ be a group scheme of finite type over an algebraically closed field $k$.
Then the following are equivalent:
\begin{enumerate}
\item $G^0_{\mathrm{red}}$ is semiabelian, that is, a torus or the extension of
  an abelian variety by a torus;
\item there is no subgroup $\Ga\hookrightarrow G$.
\end{enumerate}
\end{lemma}
\begin{proof}
  By Chevalley's Theorem \cite[Thm.~1.1]{MR1906417} there is an extension $1\to H\to G^0_{\mathrm{red}}\to
  A\to 1$ where $H$ is smooth, affine and connected and $A$ is an abelian
  variety. A subgroup $\Ga\hookrightarrow G$ would have to be contained in $H$
  which implies that $H$ is not a torus. Conversely, recall that $H(k)$ is generated by its
  semi-simple and unipotent elements by the Jordan Decomposition
  Theorem \cite[Thm.~4.4]{MR1102012}. If
  $H$ is not a torus, then there exist non-trivial unipotent elements in
  $H(k)$. But any non-trivial unipotent element of $H(k)$ lies in a subgroup
  $\Ga\hookrightarrow G$. The result follows.
\end{proof}

If $k$ is of positive characteristic, then we say that $G$ is {\badgroup}
if $G^0_{\mathrm{red}}$ is not semiabelian. We say that an algebraic stack $X$ is \emph{\badstack} if
there exists a geometric point $x$ of $X$ whose residue field $\kappa(x)$
is of characteristic $p>0$ and stabilizer group scheme $G_x$ is {\badgroup}.
In particular, the algebraic stacks $B_k\Ga$ and
$B_k\mathrm{GL}_n$ for $n>1$ are {\badstack} in positive
characteristic. The following characterization of {\badstack}
algebraic stacks will be useful.
\begin{lemma}\label{L:badgroup}
Let $X$ be a quasi-separated algebraic stack.
\begin{enumerate}
\item The stack $X$ is {\badstack} if and only if there exists a field $k$ of
  characteristic $p>0$ and a representable morphism $\phi\colon B_k\Ga \to X$.
\item If $X$ has affine stabilizers, then every representable morphism
  $\phi\colon B_k\Ga \to X$ is quasi-affine.
\end{enumerate}
\end{lemma}
\begin{proof}
  Let $k$ be an algebraically closed field and let $x\colon \spec k\to X$ be a
  geometric point with stabilizer group scheme $G$. This induces a
  representable morphism $BG\to X$. If $X$ is {\badstack}, then there exists a
  point $x$ such that $G^0_{\mathrm{red}}$ is not semiabelian.  By the previous
  lemma, there is a subgroup $\Ga\hookrightarrow G$ and hence a representable
  morphism $B\Ga\to BG$.

  Conversely, given a representable morphism $\phi\colon B_k\Ga \to X$, there
  is an induced representable morphism $\psi\colon B_k\Ga \to B_kG$. The
  morphism $\psi$ is induced by some subgroup $\Ga\hookrightarrow G$ (unique up
  to conjugation) so $X$ is {\badstack}.

  The structure
  morphism $\iota_x\colon \mathcal{G}_x\hookrightarrow X$ of the residual gerbe
  $\mathcal{G}_x$ at $x$ is quasi-affine ~\cite[Thm.~B.2]{MR2774654} and
  $\phi=\iota_x\circ\rho\circ\psi$ where $\rho\colon B_kG\to \mathcal{G}_x$ is
  affine. If $X$ has affine stabilizers, then $G$ is affine and it follows that
  the quotient $G/\Ga$ is quasi-affine since $\Ga$ is unipotent \cite[Thm.~3]{MR0130878}.
  We conclude that the morphism $\psi\colon B_k\Ga
  \to B_kG$, as well as $\phi$, is quasi-affine.
\end{proof}
We now prove Theorems~\ref{T:compact-generation} and~\ref{T:notfull}.
\begin{proof}[Proof of Theorem \ref{T:compact-generation}]
  By Lemma \ref{L:badgroup}, there exists a
  field of characteristic $p>0$ and a quasi-compact, quasi-separated and representable morphism
  $\phi\colon B_k\Ga \to X$. 

  If $\DQCOH(X)$ is compactly generated, then there is a
  compact object $M \in \DQCOH(X)$ and a non-zero map $M \to
  \RDERF \MODPSH{\phi}\mathcal{O}_{B_k\Ga}$. Indeed, 
  $\RDERF \MODPSH{\phi}\mathcal{O}_{B_k\Ga} \in \DQCOH(X)$ and is
  non-zero.  By adjunction, there is a
  non-zero map $\LDERF \phi^*M \to \mathcal{O}_{B_k\Ga}$. But the
  functor $\LDERF \phi^*$ sends compact objects of $\DQCOH(X)$ to
  compact objects of $\DQCOH(B_k\Ga)$ \cite[Ex.~3.8 \&
  Thm.~2.6(3)]{perfect_complexes_stacks}. By Proposition
  \ref{P:BGa:nocompacts}, it follows that $\LDERF \phi^*M \simeq 0$
  and we have a contradiction. Hence $\DQCOH(X)$ is not compactly
  generated.

  If $X$ has affine diagonal or is noetherian and affine-pointed, and $\DCAT(\QCOH(X))$ is compactly generated, then there is a compact object $Q \in \DCAT(\QCOH(X))$ and a non-zero map $Q \to \RDERF\QCOHPSH{\phi}\Orb_{B_k\Ga}$. Indeed, 
  $\RDERF \QCOHPSH{\phi}\mathcal{O}_{B_k\Ga} \in \DCAT(\QCOH(X))$ and is
  non-zero. Since $\QCOHPSH{\phi}$ preserves small products, $\RDERF \QCOHPSH{\phi}$ preserves small products. Furthermore, $\DCAT(\QCOH(B_k\Ga))$ is well-generated and so $\RDERF \QCOHPSH{\phi}$ admits a left adjoint $F \colon \DCAT(\QCOH(X)) \to \DCAT(\QCOH(B_k\Ga))$. By Corollary \ref{C:dsum}\eqref{CI:dsum:coprod} and \cite[Ex.~3.8]{perfect_complexes_stacks}, the
  functor $F$ preserves compact objects. Hence, Proposition \ref{P:BGa:nocompacts} implies that $F(Q) \simeq 0$. We have a contradiction and $\DCAT(\QCOH(X))$ is not compactly generated. 
\end{proof}
\begin{proof}[Proof of Theorem \ref{T:notfull}]
  By Lemma \ref{L:badgroup}, there exists a
  field of characteristic $p>0$ and a quasi-affine morphism
  $\phi\colon B_k\Ga \to X$. By Corollary \ref{C:dsum}, there exists an
  integer $n\geq 1$ such that if $N \in \QCOH(B_k\Ga)$, then
  $\RDERF \QCOHPSH{\phi}N \in \DCAT^{[0,n-1]}(\QCOH(X))$. By
  \cite[Thm.~1.1]{MR2875857}, there exists $M \in \QCOH(B_k\Ga)$ such
  that the natural map in $\DCAT(\QCOH(B_k\Ga))$:
  \[
  \bigoplus_{i\geq 0} M[in] \to \prod_{i\geq 0} M[in]
  \]
  is not a quasi-isomorphism---note that while
  \cite[Thm.~1.1]{MR2875857} only proves the above assertion in the
  case where $n=1$, a simple argument by induction on $n$ gives the
  claim above. Corollary \ref{C:dsum}\eqref{CI:dsum:qaff-cons}
  now implies that the
  natural map:
  \[
  \bigoplus_{i\geq 0} \RDERF \QCOHPSH{\phi}M[in] \to \prod_{i\geq
    0} \RDERF \QCOHPSH{\phi}M[in]
  \]
  is not a quasi-isomorphism. Since $\RDERF \QCOHPSH{\phi}M \in
  \DCAT^{[0,n-1]}(\QCOH(X))$, it follows that $\DCAT(\QCOH(X))$ is not
  left-complete. By Remark \ref{R:lc_ff}, we have established that $\Psi_X$ is not fully 
  faithful. To prove that $\Psi_X$ is not full, we will have to argue further.

  Let $L =
  \RDERF \QCOHPSH{\phi}M$, $S = \oplus_{i\geq 0} L[in]$, and
  $P=\prod_{i\geq 0}L[in]$. Also, $\Phi_X\Psi_X(S) \simeq P$ 
  (Lemma \ref{L:adjunction}\itemref{L:adjunction:aff_prop}). If $\Psi_X$ is full,
  then there exists a map $P \to S$ such that the induced map $P \to S\to \Phi_X \Psi_X(S) 
  \simeq P$ is the identity morphism. That is, $P$ is a direct summand of $S$. Since $\prod_{i\geq 0} M[in]$ is
  not bounded above \cite[Rem.~1.2]{MR2875857} and $\phi$ is
  quasi-affine, it follows that $P$ is not bounded above. But $S$ is bounded above, so $P$ 
  cannot be a direct summand of $S$; hence, we have a contradiction and $\Psi_X$ is not full.
\end{proof}
\appendix
\section{\texorpdfstring{$\DCAT_{\cm}(\ca)$}{D(A)} is well
  generated}\label{APP:local-grothendieck}

We begin with a general lemma.

\begin{lemma}\label{A.L10.1}
Let $\ct$ be a well generated triangulated category and let
$\cs\subset\ct$ be a localizing subcategory. The category $\cs$ is
well  
generated if and only if there is a set of generators in $\cs$. That
is: $\cs$ is well generated if and only if 
there is a set of objects $S\subset\cs$ such that any nonzero object
$y\in\cs$ admits a nonzero map $s\to y$, with $s\in S$.
\end{lemma}

\begin{proof}
If $\cs$ is well generated then it has a set of generators $S$
satisfying a bunch of properties, one of which is that $S$ detects
nonzero objects---see the definitions in~\cite[pp.~273-274]{MR1812507}. What needs proof is the reverse implication.

Suppose therefore that $\cs$ contains a set of objects $S$ as in the
Lemma, that is every nonzero object $y\in\cs$ admits a nonzero map 
$s\to y$ with $s\in S$. By
\cite[Prop.~8.4.2]{MR1812507} the set $S$
is contained in $\ct^\alpha$ for some regular cardinal $\alpha$.
If $\cl=\Loc(S)$ is the localizing subcategory
generated by $S$ then \cite[Thm.~4.4.9]{MR1812507}
informs us that $\cl$ is well generated. Since $S
\subset\cs$ and $\cs$ is localizing it follows that
$\cl\subset\cs$.

We know that $\cl$ is well generated; to finish the proof
it suffices to show that the inclusion $\cl\subset\cs$
is an equality. In any case the inclusion is a coproduct-preserving 
functor from the well generated category $\cl$ and must have a right adjoint.
For every object $y\in \cs$, \cite[Prop.~9.1.8]{MR1812507}  tells us
that
there is a triangle
in $\cs$
\[
\xymatrix{
x \ar[r] & y \ar[r] & z\ar[r] & \T x
}
\]
with $x\in\cl$ and $z\in\cl^\perp$. Since $z\in\cl^\perp\subset S^\perp$ we have that 
the morphisms $s\to z$, with $s\in S$, all vanish. By the hypothesis
of the Lemma it follows that $z=0$, and hence $y\cong x$ belongs to $\cl$.
\end{proof}

\begin{remark}\label{A.R10.2}
We specialize Lemma~\ref{A.L10.1}
to the situation where $\ct=\DCAT(\ca)$ is the
derived category of a Grothendieck abelian category $\ca$; by
\cite[Thm.~0.2]{MR1874232} we know that 
$\ct$ is well generated, and Lemma~\ref{A.L10.1} informs us that
a localizing subcategory of $\DCAT(\ca)$ is well generated if and 
only if it has a set of generators.
\end{remark}

Let $\ca$ be an abelian category and fix a fully faithful subcategory $\cc \subseteq \ca$. Following 
\cite[Tag \spref{02MO}]{stacks-project} we say that
\begin{enumerate}
\item $\cc$ is a \fndefn{Serre} subcategory if it is non-empty and if
  $C_1 \to A \to C_2$ is an exact sequence in $\ca$ with
  $C_1, C_2 \in \cc$, then $A \in \cc$;
\item $\cc$ is a \fndefn{weak Serre} subcategory if it is non-empty and if
  \[ 
  \xymatrix{C_1 \ar[r] & C_2 \ar[r] & A \ar[r] & C_3 \ar[r] & C_4,}
  \]
  is an exact sequence in $\ca$, where the $C_i\in \cc$ and $A\in \ca$, then $A \in \cc$.
\end{enumerate}
Clearly, Serre subcategories are weak Serre subcategories. Also, weak Serre subcategories are automatically abelian and the inclusion $\cc \subseteq \ca$ is exact \cite[Tags \spref{02MP} \& \spref{0754}]{stacks-project}. Moreover, the subcategory $\DCAT_{\cc}(\ca)$ of $\DCAT(\ca)$, consisting of complexes in $\ca$ with cohomology in $\cc$, is triangulated \cite[Tag \spref{06UQ}]{stacks-project}.

The main result of this appendix is

\begin{theorem}\label{T10.1}
Let $\ca$ be a Grothendieck abelian category and let
$\cm\subseteq\ca$ be a weak Serre subcategory closed under coproducts. If $\cm$ is Grothendieck abelian, then $\DCAT_\cm^{}(\ca)$ is well
generated.
\end{theorem}

The example we have in mind is where $X$ is an algebraic stack,
$\ca$ is the category of lisse-\'etale sheaves of $\Orb_X$-modules, and 
$\cm$ is the subcategory of quasi-coherent sheaves.

If $\lambda$ is a cardinal and $\cb$ is a cocomplete category, then we let $\cb^\lambda$ 
denote the subcategory of $\lambda$-presentable objects. If $\cb$ is locally presentable, 
then $\cb^\lambda$ is always a set.

It is clear that $\DCAT_\cm^{}(\ca)$ is a localizing subcategory of 
the well generated triangulated category $\DCAT(\ca)$;
Remark~\ref{A.R10.2}
tells us that to prove Theorem~\ref{T10.1} it suffices to exhibit a
set $S$ of generators for $\DCAT_\cm^{}(\ca)$. The idea is simple enough:
we will find a cardinal $\lambda$ such that 
$S=\DCAT_{\cm^\lambda}^-(\ca^\lambda)\subset \DCAT_\cm^{}(\ca)$, which is
obviously essentially small, suffices.
Thus, the problem becomes to better understand the category of $\lambda$-presentable 
objects in $\ca$. The results 
below are easy to extract from
\cite{MR3212862}, but for the reader's convenience we give a self-contained
treatment.

\begin{remark}\label{A.R10.5}
Let $\ca$ be a Grothendieck abelian category. 
By the Gabriel--Popescu 
theorem, there exists a ring $R$ and a pair of adjoint additive functors
\[
\xymatrix{
F\colon\MOD(R)\ar@<0.5ex>[r] &\ca\cocolon G\ar@<0.5ex>[l],
}\]
such that $F$ is exact and $FG\simeq \id$. Let $\mu$ be an infinite cardinal
$\geq$ to the cardinality of $R$.
\end{remark}

\begin{lemma}\label{A.L10.6}
With notation as in Remark~\ref{A.R10.5}, let $\lambda>\mu$ be a regular
cardinal. Then the $\lambda$-presentable objects of $\ca$ are precisely
the objects of $\ca$ isomorphic to some $FN$, where $N$ is an $R$-module
of cardinality $<\lambda$.
\end{lemma}

\begin{proof}
Let us first prove that, if $N$ is an $R$-module
of cardinality $<\lambda$, then $FN$ is $\lambda$-presentable.
 Suppose
 $\{x_i,\,i\in I\}$ is a $\lambda$-filtered system
in $\ca$,
and suppose that in
the category $\ca$ we are given a map 
$\phi\colon FN\longrightarrow \colim x_i$.
We need to show that $\phi$ factors through some $x_i$. In the category
of $R$-modules, there is a natural map 
\[\xymatrix@C+20pt{
\colim Gx_i \ar[r]^\rho & G\big[\colim x_i\big].
}\] 
Since $F$ respects colimits and $FG\simeq \id$, the map
$F\rho$ is an isomorphism.
As $F$ is exact it must annihilate the kernel and cokernel of $\rho$. 
Form the pullback square
\[\xymatrix@C+20pt{
P\ar[r]^{\phi}\ar[d] &  N\ar[d]^\theta\\
\colim Gx_i \ar[r]^\rho & G\big[\colim x_i\big].
}\]
The image of $\phi$ is a submodule 
of $N$, hence has cardinality $<\lambda$.
 If we
lift every element of $\text{Image}(\phi)$ arbitrarily to an element of $P$, 
the lifts generate
a submodule $M\subset P$ of cardinality $<\lambda$. 
The kernel (respectively cokernel) of the  map
 $M\to N$ is a submodule of $\text{Kernel}(\rho)$
(respectively $\text{Coker}(\rho)$), and hence both are
annihilated by $F$. Summarizing: 
we have produced a morphism $M\to N$ of $R$-modules, with
$M$ of cardinality $<\lambda$ and $FM\to FN$ 
an isomorphism in $\ca$, and such that the composite
$M\to N\longrightarrow G\big[\colim x_i\big]$ factors through 
$\colim Gx_i$. 
But $\{Gx_i,\,i\in I\}$ is a 
$\lambda$-filtered system in $\MOD(R)$ and $M$ is of cardinality $<\lambda$,
and hence the map factors as $M\to Gx_i$ for some $i\in I$.

It remains to prove the converse: suppose $a\in\ca$ is a $\lambda$-presentable
 object. Then $Ga$ is an $R$-module, hence it is the $\lambda$-filtered 
colimit of all its submodules $N_i$ of cardinality $<\lambda$. But then the 
identity map $a\to a$ is a map from the $\lambda$-presentable
object $a$ to the $\lambda$-filtered 
colimit of the $FN_i$, and therefore factors through some $FN_i$.
Thus $a$ is a direct summand of an object $FN_i$ where
the cardinality of $N_i$ is $<\lambda$. On the other hand the map
$N_i\to Ga$ is injective, hence so is $FN_i\to FGa=a$. Thus $a\cong FN_i$.
\end{proof}

\begin{lemma}\label{A.L10.7}
Let $\ca$ be a Grothendieck abelian category. There is an
infinite  cardinal $\nu$ with the following properties: if $\lambda\geq\nu$ is a regular cardinal, then 
\begin{enumerate}
\item $\ca^\lambda$ is a Serre subcategory of $\ca$;
\item every object of $\ca$ is a $\lambda$-filtered colimit of subobjects
belonging to $\ca^\lambda$;
\item an object belongs to $\ca^\lambda$ if and
only if it is the quotient of a coproduct of $<\lambda$ objects of $\ca^\nu$; and
\item any pair of morphisms $x\to y\leftarrow n$ in $\ca$, where $x\to y$ is epi and
$n\in\ca^\lambda$, may be completed to a commutative square
\[\xymatrix@C+20pt{
m\ar@{->>}[r]\ar[d] & n\ar[d]\\
x\ar@{->>}[r] & y
}\]
with $m\in\ca^\lambda$ and $m\to n$ epi. Moreover, if $n \to y$ is mono, then $m \to x$ can be chosen to be mono.
\end{enumerate}
\end{lemma}

\begin{proof}
Let $\nu=\mu+1$ be the successor of the infinite cardinal $\mu$ of
Remark~\ref{A.R10.5}. By Lemma~\ref{A.L10.6} the objects of $\ca^\lambda$ are
precisely the ones isomorphic to $FM$ where $M$ is of cardinality $<\lambda$.

For (1), it is readily verified that a subobject (resp.\ a quotient) of $FM$
can be expressed as $FN$ where $N$ is a submodule (resp.\ a quotient) of
$M$. This shows that $\ca^\lambda$ is closed under taking subobjects and
quotients; we will later see that it is also closed under extensions.

For (2), if $a$ is an object of $\ca$ then $Ga$ is
the $\lambda$-filtered colimit of all the submodules $M_i\subset Ga$ of
cardinality $<\lambda$, and $a\cong FGa$ is the $\lambda$-filtered colimit
of $FM_i\in\ca^\lambda$. 

For (3), observe that any coproduct
of $<\lambda$ objects in $\ca^\lambda$ belongs to $\ca^\lambda$, and
if $M$ is a module of cardinality $<\lambda$ then $M$ is a quotient
of the free module on all its elements, which is a coproduct of $<\lambda$
copies of $R$. Thus $FM$ is the quotient of a coproduct of $<\lambda$
copies of $FR\in\ca^\nu$.

For (4), let $x\rightarrow y\leftarrow n$ be a pair of morphisms in $\ca$,
with $x\to y$ epi and
$n\in\ca^\lambda$. Let $\tilde{m}$ be the pullback of $n \to y$ along $x\to y$. It is sufficient to find a 
subobject $m$ of $\tilde{m}$ belonging to $\ca^\lambda$ such that $m \to \tilde{m} \to n$ is epi. By 
(2), we may express $\tilde{m}=\colim m_i$ as a $\lambda$-filtered colimit of subobjects belonging 
$\ca^\lambda$. If $n_i \subseteq n$ is the image of $m_i$ in $n$, then (1) implies that $n_i \in \ca^\lambda$. Since $n \in \ca^\lambda$, there is an $i$ such that $n_i=n$. Taking $m=m_i$ does the job. By construction, if $n \to y$ is mono, then $m \to x$ is mono. 

Finally, to show that $\ca^\lambda$ is closed under extensions, we note that if $0 \to k \to x \to n \to 0$ is an exact sequence with $k$, $n\in \ca^\lambda$, then (4) implies that there is a suboboject $m$ of $x$ such that $m\in \ca^\lambda$ and $m \to n$ is epi. It follows that $k\oplus m \to x$ is epi and consequently, $x\in \ca^\lambda$, as required.
\end{proof}

\begin{proof}[Proof of Theorem~\ref{T10.1}]
Because $\cm$ and $\ca$ are both Grothendieck abelian categories we may choose
regular cardinals $\nu$ for $\cm$ and $\nu'$ for $\ca$ as in
Lemma~\ref{A.L10.7}. The category $\cm^\nu$ is an essentially small subcategory
of $\ca$, hence must be contained in $\ca^\beta$ for some 
regular cardinal $\beta$. Let
$\lambda$ be a regular cardinal $>\max(\beta,\nu,\nu')$.
By construction $\cm^\nu\subset\ca^\lambda$, and by Lemma~\ref{A.L10.7} every object
in $\cm^\lambda$ is the quotient of a coproduct of $<\lambda$
objects in $\cm^\nu$. Hence $\cm^\lambda\subset\ca^\lambda$. But since
every object in $\cm\cap\ca^\lambda$ is $\lambda$-presentable in $\ca$ 
it must be $\lambda$-presentable in the smaller $\cm$, and we conclude
that $\cm\cap\ca^\lambda=\cm^\lambda$.

We have now made our choice of $\lambda$ and we let $\cb=\ca^\lambda$. 
By Remark~\ref{A.R10.2} it 
suffices to show that,
  given any non-zero object $Z\in\DCAT_\cm^{}(\ca)$, there is an object
  $N\in\DCAT_{\cb\cap\cm}^-(\cb)$ and a non-zero map $N \to Z$.
If $Z$ is the chain complex
\[
\xymatrix{
\cdots \ar[r] & Z^{i-1} \ar[r]^-{\partial} & Z^i \ar[r]^-{\partial}&
Z^{i+1} \ar[r] & \cdots,}
\]
we let $Y^i\subseteq Z^i$ be the cycles,
in other words the kernel of $\partial\colon Z^i \to Z^{i+1}$,
and $X^i\subseteq Y^i$ be the boundaries,
that is the image of $\partial\colon Z^{i-1} \to Z^i$.
We are assuming that
 $Z\in\DCAT_\cm^{}(\ca)$ is non-zero, meaning its cohomology is not all
zero; without loss of generality we may assume $H^0(Z)\neq0$.
Thus $Y^0/X^0$ is a non-zero object of $\cm$.

By Lemma~\ref{A.L10.7}, applied to $\cb\cap\cm=\cm^\lambda\subset\cm$,
 the object
$Y^0/X^0\in\cm$ is a $\lambda$-filtered 
colimit of its subobjects belonging 
to $\cb\cap\cm$; since  $Y^0/X^0\neq0$ we may choose a subobject
$M\subseteq Y^0/X^0$, with $M\in\cb\cap\cm$ and $M\neq0$. 
By Lemma~\ref{A.L10.7},
applied to the pair of maps $Y^0 \to Y^0/X^0\leftarrow M$ in $\ca$,
we may complete to a commutative
square
\[\xymatrix@C+20pt{
N^0\ar@{->>}[r]^\phi\ar[d] & M\ar[d]\\
Y^0\ar@{->>}[r] & Y^0/X^0
}\]
with $N^0\in\cb$.
Since $Y^0$ is the
kernel of $Z^0 \to Z^1$ this gives us a commutative square
\[\xymatrix{
 N^0\ar[r]\ar[d] & 0\ar[d] \\
Z^0\ar[r] & Z^1
}\]
such that the 
image of the map $N^0 \to Y^0/X^0=H^0(Z)$ 
is non-zero and belongs
to $\cb\cap\cm$.

We propose to inductively extend this to the left.
We will define a commutative diagram
\[\xymatrix@-1ex{
 & & N^i \ar[r]\ar[d]& N^{i+1} \ar[r]\ar[d]& \cdots\ar[r]&
 N^{-1}\ar[d]\ar[r]& N^0\ar[r]\ar[d] & 0\ar[r]\ar[d] &\cdots\\
\cdots \ar[r]& Z^{i-1} \ar[r] & Z^i \ar[r]& Z^{i+1} \ar[r]& \cdots\ar[r]&
 Z^{-1}\ar[r]& Z^0\ar[r] & Z^1\ar[r] &\cdots
}\]
where
\begin{enumerate}
\item
 The subobjects $N^j \subseteq Z^j$ belong to $\cb$.
\item
For $j>i$ the cohomology of $N^{j-1} \to N^j \to N^{j+1}$ belongs
to $\cb\cap\cm$.
\item
Let $K^i$ be the kernel of the map $N^i \to N^{i+1}$. Then
the image of the natural map $K^i \to H^i(Z)$ belongs
to $\cb\cap\cm$.
\setcounter{enumiii}{\value{enumi}}
\end{enumerate}
Since we have constructed $N^0$ we only need to prove
the inductive step. Let us therefore suppose we have constructed
the diagram as far as $i$; we need to extend it to $i-1$.
We first form the pullback square
\[\xymatrix@C+20pt{
L^i\ar[r]\ar[d] & K^i\ar[d]\\
X^i\ar[r] & Y^i
}\]
Since $X^i\to Y^i$ and $K^i\to Y^i$ are monomorphisms so are $L^i\to K^i$
and $L^i\to X^i$. Since 
$N^i$ belongs to $\cb$ so do its subobjects $L^i\subset K^i$. The
cokernel of $L^i\to K^i$ is the image of $K^i\to Y^i/X^i=H^i(Z)$,
and belongs to $\cb\cap\cm$ by (3).
Next we apply Lemma~\ref{A.L10.7} to the pair of maps
$Z^{i-1}/X^{i-1}\to X^i\leftarrow L^i$ in $\ca$, completing to
a commutative square
\[\xymatrix@C+20pt{
M^i\ar@{->>}[r]\ar[d] & L^i\ar[d]\\
Z^{i-1}/X^{i-1}\ar@{->>}[r] & X^i
}\]
with $M^i\in\cb$. Form the pullback
\[\xymatrix@C+20pt{
\widetilde M^i\ar[r]\ar[d] & M^i\ar[d]\\
Y^{i-1}/X^{i-1}\ar[r] & Z^{i-1}/X^{i-1}
}\]
Since $Y^{i-1}/X^{i-1}\to Z^{i-1}/X^{i-1}$ is injective so is 
$\widetilde M^i\to M^i$, making
$\widetilde M^i$ a subobject of $M^i\in\cb$. Hence $\widetilde
M^i$ belongs to $\cb$. But
now the map $\widetilde M^i\to Y^{i-1}/X^{i-1}=H^{i-1}(Z)$ is a morphism
from the $\lambda$-presentable object $\widetilde M^i\in\cb=\ca^\lambda$ to
the object $H^{i-1}(Z)\in\cm$, which by Lemma~\ref{A.L10.7} is a
$\lambda$-filtered colimit of its subobjects in $\cm^\lambda=\cb\cap\cm$.
Hence the map $\widetilde M^i\to Y^{i-1}/X^{i-1}$ factors as 
$\widetilde M^i\to P^i\to Y^{i-1}/X^{i-1}$ with $P^i\in\cb\cap\cm$
a subobject of $Y^{i-1}/X^{i-1}$. 
Form the pushout square in $\cb$
\[\xymatrix@C+20pt{
\widetilde M^i\ar[r]\ar[d] & M^i\ar[d]\\
P^i\ar[r] & Q^i
}\]
and let $Q^i\to Z^{i-1}/X^{i-1}$ be the natural map. We have a commutative square
\[\xymatrix@C+20pt{
Q^i\ar@{->>}[r]^\phi\ar[d] & L^i\ar[d]\\
Z^{i-1}/X^{i-1}\ar@{->>}[r] & X^i
}\]
and the kernel of $\phi$ maps isomorphically to the subobject
$P^i\subset H^{i-1}(Z)$, with $P^i\in\cb\cap\cm$. Finally apply
Lemma~\ref{A.L10.7} to the pair of maps $Z^{i-1}\to Z^{i-1}/X^{i-1}\leftarrow
Q^i$ to complete to a square
\[\xymatrix@C+20pt{
N^{i-1}\ar@{->>}[r]\ar[d] & Q^i\ar[d]\\
Z^{i-1}\ar@{->>}[r] &Z^{i-1}/X^{i-1} 
}\]
with $N^i\in\cb$. We leave it to the reader to check that
the diagram
\[\xymatrix@-1ex{
 & & N^{i-1} \ar[r]\ar[d]& N^{i} \ar[r]\ar[d]& \cdots\ar[r]&
 N^{-1}\ar[d]\ar[r]& N^0\ar[r]\ar[d] & 0\ar[r]\ar[d] &\cdots\\
\cdots \ar[r]& Z^{i-2} \ar[r] & Z^{i-1} \ar[r]& Z^{i} \ar[r]& \cdots\ar[r]&
 Z^{-1}\ar[r]& Z^0\ar[r] & Z^1\ar[r] &\cdots
}\]
satisfies hypotheses (1), (2) and (3) of our induction.
\end{proof}

\section{\texorpdfstring{$\DQCOH(X)$}{D(X)} is left-complete}\label{APP:left-complete}
In this section we prove the following Theorem.
\begin{theorem}\label{APP:T:left-complete}
  If $X$ is an algebraic stack, then $\DQCOH(X)$ is well generated. In
  particular, it admits small products. Moreover, $\DQCOH(X)$ is left-complete.
\end{theorem}
\begin{proof}
  The subcategory $\QCOH(X) \subseteq \MOD(X_{\lisset})$ is weak Serre and the inclusion 
  is coproduct preserving. Since
  $\QCOH(X)$ and $\MOD(X_{\lisset})$ are Grothendieck
  abelian categories
  \cite[\href{http://stacks.math.columbia.edu/tag/07A5}{07A5} \&
  \href{http://stacks.math.columbia.edu/tag/0781}{0781}]{stacks-project},
  it follows that $\DQCOH(X)$ is well generated (Theorem
  \ref{T10.1}). By \cite[Cor.~1.18]{MR1812507}, $\DQCOH(X)$ admits small
  products.

  It remains to prove that $\DQCOH(X)$ is left-complete. Let $p\colon U
  \to X$ be a smooth surjection from an algebraic space $U$. Let
  $U^+_{\bullet,\mathrm{\acute{e}t}}$ denote the resulting strictly
  simplicial algebraic space \cite[\S 4.1]{MR2312554}. By
  \cite[Ex.~2.2.5]{MR2434692}, there is an equivalence of triangulated
  categories $\DQCOH(X) \simeq
  \DQCOH(U^+_{\bullet,\mathrm{\acute{e}t}})$. The inclusion
  $\QCOH(U^+_{\bullet,\mathrm{\acute{e}t}}) \subseteq
  \MOD(U^+_{\bullet,\mathrm{\acute{e}t}})$ is exact, stable under
  extensions, and coproduct preserving. It follows that the functor
  $\omega\colon \DQCOH(U^+_{\bullet,\mathrm{\acute{e}t}}) \to
  \DCAT(U^+_{\bullet,\mathrm{\acute{e}t}})$ is exact and coproduct
  preserving. As we already have seen, the category $\DQCOH(X) \simeq
  \DQCOH(U^+_{\bullet,\mathrm{\acute{e}t}})$ is well generated. Thus the
  functor $\omega$ admits a right adjoint $\lambda$
  \cite[Prop.~1.20]{MR1812507}. Because the functor $\omega$ is fully
  faithful, the adjunction $\mathrm{id} \Rightarrow \lambda\circ
  \omega$ is an isomorphism of functors. 

  Note that because $\lambda$ is a right adjoint, it preserves
  products. In particular, it remains to prove that if $K \in
  \DQCOH(U^+_{\bullet,\mathrm{\acute{e}t}})$,
  then there exists a distinguished triangle in
  $\DCAT(U^+_{\bullet,\mathrm{\acute{e}t}})$ (where we also take the
  products in $\DCAT(U^+_{\bullet,\mathrm{\acute{e}t}})$):
  \[
  \xymatrix{\omega(K) \ar[r] & \prod_{n\geq 0} \tau^{\geq -n} \omega(K)
    \ar[r]^-{\mathrm{1-shift}} & \prod_{n\geq 0} \tau^{\geq
      -n}\omega(K) \ar[r] & \omega(K)[1].}
  \]
  Indeed, this follows from the observation that $\tau^{\geq
    -n}\omega(K) \simeq \omega(\tau^{\geq -n}K)$ for all integers $n$ and $K \to
  \lambda\circ \omega(K)$ is an isomorphism. 

  Let $(W \to U_n)$ be an object of
  $U^+_{\bullet,\mathrm{\acute{e}t}}$. The resulting slice
  $U^+_{\bullet,\mathrm{\acute{e}t}}/(W \to U_n)$ is equivalent to the
  small \'etale site on $W$. In particular, it follows that if $W$ is
  an affine scheme and $M\in \QCOH(U^+_{\bullet,\mathrm{\acute{e}t}})
  \cong \QCOH(X)$, then $H^p(U^+_{\bullet,\mathrm{\acute{e}t}}/(W \to
  U_n),M) = 0$ for all $p>0$
  \cite[\href{http://stacks.math.columbia.edu/tag/01XB}{01XB} \&
  \href{http://stacks.math.columbia.edu/tag/0756}{0756}]{stacks-project}. Now
  let $\mathcal{B} \subseteq U^+_{\bullet,\mathrm{\acute{e}t}}$ denote
  the full subcategory consisting of those objects $(W \to U_n)$,
  where $W$ is an affine scheme. It follows that $\mathcal{B}$ satisfies the
  requirements of
  \cite[\href{http://stacks.math.columbia.edu/tag/08U3}{08U3}]{stacks-project}
  and we deduce the result. 
\end{proof}
\section{The bounded below derived category}
In this section, we prove an analog of \cite[Cor.~II.7.19]{MR0222093} for noetherian 
algebraic stacks that are affine-pointed. Essentially for free, we will also establish Lurie's 
result \cite[Thm.~3.8]{lurie_tannaka}.
\begin{theorem}\label{T:bdd_below}
  Let $X$ be an algebraic stack. If $X$ is either quasi-compact with affine diagonal or 
  noetherian and affine-pointed, then the natural functor
  \[
 \Psi_X^+\colon \DCAT^+(\QCOH(X)) \to \DQCOH^+(X)
  \]
  is an equivalence.
\end{theorem}
For noetherian algebraic spaces, a version of Theorem \ref{T:bdd_below} for the 
unbounded derived category was proved in \cite[Tag \spref{09TN}]{stacks-project} and we 
will closely follow this approach. The following two lemmas do most of the work.
\begin{lemma}[cf.\ {\cite[Tag \spref{09TJ}]{stacks-project}}]\label{L:inj_stk}
  Let $X$ be a quasi-compact and quasi-separated algebraic stack and let $I$ be an injective 
  object of  $\QCOH(X)$.
  \begin{enumerate}
  \item \label{L:inj_stk:aff} Then $I$ is a direct 
    summand of $p_*J$, where $p\colon \spec A \to X$ is smooth and surjective and 
    $J$ is an injective $A$-module.
  \item \label{L:inj_stk:noeth} If $X$ is noetherian, then $I$ is a direct 
    summand of a filtered 
    colimit $\colim_i {F}_i$ of quasi-coherent sheaves of the form
    ${F}_i = \gamma_* G_i$, where $\gamma\colon Z_i\to X$ is a morphism
    from an artinian scheme $Z_i$ and ${G}_i \in \COH(Z_i)$.
  \end{enumerate}
\end{lemma}
\begin{proof}
  Let $p\colon U \to X$ be a smooth and surjective morphism, where $U=\spec A$ is an affine scheme. 
  Let $I$ be an injective object of $\QCOH(X)$. Choose an injective object $J$ 
  of $\QCOH(U)$ and an injection $p^*I \subseteq J$. By adjunction, we have an inclusion $I 
  \subseteq p_*J$. Since $p^*$ is exact, $p_*J$ is injective in 
  $\QCOH(X)$ and $I$ is a direct summand of $p_*J$. This proves 
  \itemref{L:inj_stk:aff}. For \itemref{L:inj_stk:noeth}: we may now reduce to the 
  case where $X=U$. The result is now well-known (e.g., \cite[Tag 
  \spref{09TI}]{stacks-project}).
\end{proof}
\begin{lemma}[cf.\ {\cite[Tag \spref{09TL}]{stacks-project}}]\label{L:bdd_below_coho}
  Let $X$ be an algebraic stack and let $I$ be an injective object of 
  $\QCOH(X)$. If $X$ is quasi-compact with affine diagonal (resp.\ noetherian and 
  affine-pointed), then 
  \begin{enumerate}
    \item \label{L:bdd_below_coho:global} $H^q(U_{\lisset},I) = 0$ for every $q>0$ and smooth morphism $u\colon U \to X$ that is affine (resp.\ has affine fibers);
    \item \label{L:bdd_below_coho:local} for any morphism $f\colon X \to Y$ of algebraic stacks, where
      $Y$ has affine diagonal (resp.\ $Y$ is affine-pointed) we have
      $\RDERF^q(f_{\lisset})_*I = 0$ for $q>0$.
    \end{enumerate}
\end{lemma}
\begin{proof}
  Let $W$ be an affine (resp.\ artinian) scheme and let $M \in \QCOH(W)$ be injective 
  (resp.\ $M\in \COH(W)$). Let $w \colon W \to X$ be a smooth and surjective morphism 
  (resp.\ a morphism). By Lemma \ref{L:inj_stk}, it is sufficient to prove the result for 
  $I=w_*M$. Since $X$ has affine diagonal (resp.\ $X$ is affine-pointed), 
  $w$ is affine. In particular, the natural map $(w_*M)[0] \to \RDERF (w_{\lisset})_*M$ is a 
  quasi-isomorphism.

  We now prove \itemref{L:bdd_below_coho:global}. Let $u_W \colon W_U \to W$ be the 
  pull back of $u$ along $w$ and let $w_U \colon W_U \to U$ be the pull back of $w$ along 
  $u$. In both cases, $u_W$ is smooth and affine and $w_U$ is affine; in particular, $W_U$ 
  is an affine scheme. Since $u$ is smooth,
  \begin{align*}
  \RDERF\Gamma(U_{\lisset},I) &\simeq \RDERF\Gamma(U_{\lisset},u^*\RDERF (w_{\lisset})_*M)\simeq \RDERF\Gamma(U_{\lisset},\RDERF((w_U)_{\lisset})_*(u_W^*M))\\
    &\simeq \RDERF\Gamma((W_U)_{\lisset},M).
  \end{align*}
  The result now follows from the affine case (e.g., 
  \cite[III.1.3.1]{EGA}). 

  For \itemref{L:bdd_below_coho:local}: let $v\colon V \to Y$ be a smooth 
  morphism, where $V$ is an affine scheme. Since $Y$ has affine diagonal 
  (resp.\ is affine-pointed), $v$ is affine (resp.\ has affine fibers). By  
  \itemref{L:bdd_below_coho:global}, $H^q((V\times_Y X)_{\lisset},I) =0$. But 
  $\RDERF^q(f_{\lisset})_*I$ is the sheafification of the presheaf $V\mapsto H^q((V\times_Y 
  X)_{\lisset},I)$; the result follows.
\end{proof}
\begin{proof}[Proof of Theorem \ref{T:bdd_below}]
  We first establish the full faithfulness. Let $F$, $G\in \DCAT^+(\QCOH(X))$; then we wish 
  to prove that the natural map
  \[
  \Hom_{\DCAT(\QCOH(X))}(F,G) \to \Hom_{\DCAT(X)}(F,G)
  \]
  is an isomorphism. A standard way-out argument shows that it is sufficient to prove that 
  the natural map
  \[
  \Ext^q_{\QCOH(X)}(N,M) \to \Ext^q_{\Orb_X}(N,M)
  \]
  is an isomorphism for every $q\in \Z$ and $M$, $N\in \QCOH(X)$. For $q<0$ both sides 
  vanish and for $q=0$ we clearly have an isomorphism. For $q>0$, since every $M$ 
  embeds in a quasi-coherent injective $I$, a standard $\delta$-functor argument shows 
  that it is sufficient to prove that if $I$ is an injective object of $\QCOH(X)$, then 
  \[
  \Ext^q_{\Orb_{X}}(N,I) = 0
  \]
  for all $q>0$ and $N\in \QCOH(X)$. To see this we note that by Lemma 
  \ref{L:inj_stk}\itemref{L:inj_stk:aff}, $I$ is a direct summand of $(p_{\QCOH})_*J$, where 
  $p\colon \spec A \to X$ is smooth and surjective and $J$ is an injective $A$-module. Thus, it 
  suffices to prove the result when $I=(p_{\QCOH})_*J$. By Lemma
  \ref{L:bdd_below_coho}\itemref{L:bdd_below_coho:local}, the natural map 
  $((p_{\QCOH})_*J)[0] \to \RDERF(p_{\lisset})_*J$ is a quasi-isomorphism. Hence, there are 
  natural isomorphisms:
  \[
  \Ext^q_{\Orb_{X}}(N,(p_{\QCOH})_*J) \cong \Ext^q_{\Orb_{X}}(N,\RDERF(p_{\lisset})_*J) \cong 
  \Ext^q_{\Orb_{\spec A}}(p^*N,J).
  \]
  We are now reduced to the affine case, which is well-known (e.g., 
  \cite[Lem.~5.4]{MR1214458}).

  For the essential surjectivity, we argue as follows: by induction and using the full 
  faithfulness, one easily sees that $\DCAT^b(\QCOH(X)) \simeq \DQCOH^b(X)$. Passing to 
  homotopy colimits, we obtain the claim.
\end{proof}
\bibliographystyle{bibstyle}
\bibliography{references,local_references}

\end{document}